\documentclass[a4paper, reqno, 11pt]{amsart}

\usepackage{cite}
\usepackage{amsmath}
\usepackage{amssymb}
\usepackage{amsthm}   
\usepackage{empheq}   
\usepackage{graphicx}
\usepackage{color}
\usepackage{appendix}
\usepackage{hyperref} 
\usepackage{mathrsfs}

\newtheorem{theorem}{Theorem}[section]
\newtheorem{lemma}[theorem]{Lemma}
\newtheorem{proposition}[theorem]{Proposition}
\newtheorem{corollary}[theorem]{Corollary}

\theoremstyle{definition}
\newtheorem{definition}[theorem]{Definition}
\newtheorem{example}[theorem]{Example}

\theoremstyle{remark}
\newtheorem{remark}[theorem]{Remark}

\numberwithin{equation}{section}

\def\EE{{\mathcal{E}}}
\def\FF{\mathcal{F}}

\begin{document}

\title{Killing transform on regular Dirichlet subspaces}

\author{ Liping Li}
\address{School of Mathematical Sciences, Fudan University, Shanghai 200433, China.}
\email{lipingli10@fudan.edu.cn}

\author{Jiangang Ying}
\address{School of Mathematical Sciences, Fudan University, Shanghai 200433, China.}
\email{jgying@fudan.edu.cn}
\thanks{Research supported in part by NSFC grant 11271240.}





\begin{abstract}
In this paper, we shall consider the killing transform induced by a multiplicative functional on regular Dirichlet subspaces of a fixed Dirichlet form. Roughly speaking, a regular Dirichlet subspace is a closed subspace with Dirichlet and regular properties of fixed Dirichlet space. By using the killing transforms, our main results indicate that the big jump part of fixed Dirichlet form is not essential for discussing its regular Dirichlet subspaces. This fact is very similar to the status of killing measure when we consider the questions about regular Dirichlet subspaces in \cite{LY15-2}. 
\end{abstract}

\maketitle

\section{Introduction}\label{SEC1}

The killing transform induced by a multiplicative functional (MF for short) is a very important transform in the theory of Markov processes. We refer the  introduction to multiplicative functionals and killing transform to \cite{BG68}. In the context of nearly symmetric Markov processes, the second author successfully described the associated Dirichlet form of Markov process after killing transform via bivariate Revuz measure in \cite{Y96}. Then we extended the results of \cite{Y96} to semi-Dirichlet forms in \cite{LY15}. Roughly speaking, fix a nearly symmetric Markov process $X$ whose state space and reference measure are denoted by $E$ and $m$, its associated (non-symmetric) Dirichlet form on $L^2(E,m)$ is $(\EE,\FF)$. Let $M$ be an exact multiplicative functional of $X$ and $\nu_M$ its bivariate Revuz measure relative to $m$. Then the subprocess $(X,M)$ of $X$ killed by $M$ is still nearly symmetric on $E$, and its associated Dirichlet form on $L^2(E,m)$ is given by
\begin{equation}\label{EQ1FMF}
\begin{aligned}
	&\FF^M=\FF\cap L^2(E,\rho_M)\cap L^2(E,\lambda_M) \\
	&\EE^M(u,v)=\EE(u,v)+\nu_M(u\otimes v),\quad u,v\in \FF^M,
\end{aligned}\end{equation}
where $\rho_M$ and $\lambda_M$ stand for the left and right marginal measures of $\nu_M$ respectively, and $u\otimes v(x,y):=u(x)v(y)$ for any $x,y\in E$. 

On the other hand, the second author with his co-authors first introduced a new conception in the theory of Dirichlet forms, say regular Dirichlet subspace, in 2003 and then characterized all regular Dirichlet subspaces of 1-dimensional Brownian motion in \cite{FFY05}. To introduce this conception, we refer the background of Dirichlet forms and relevant potential theories to \cite{CF12} and \cite{FOT11}. Let $E$ be a locally compact separable metric space and $m$ a fully supported Radon measure on $E$. Given two regular Dirichlet forms $(\EE,\FF)$ and $(\EE',\FF')$ on $L^2(E,m)$, if 
\[
	\FF'\subset \FF, \quad \EE(u,v)=\EE'(u,v),\quad u,v\in \FF',
\]
then $(\EE',\FF')$ is called a \emph{regular Dirichlet subspace} (also \emph{regular subspace} in abbreviation) of $(\EE,\FF)$. We usually use $(\EE',\FF')\prec (\EE,\FF)$ to represent that $(\EE',\FF')$ is a regular subspace of $(\EE,\FF)$. 
Recently in our last paper \cite{LY15-2}, we together explored the basic structure of regular Dirichlet subspaces for general Dirichlet forms, in which some celebrated process transforms are employed. Particularly, the killing transform induced by a positive continuous additive functional and its inverse transform, i.e. the resurrected transform, indicate that the killing part of a Dirichlet form is not essential for its regular subspaces. More precisely, if $(\EE,\FF)$ is a regular Dirichlet form on $L^2(E,m)$ whose killing measure is not zero, then we may write its resurrected Dirichlet form $(\EE^\mathrm{res},\FF^\mathrm{res})$ through the technique in Theorem~5.2.17 of \cite{CF12}. Naturally, $(\EE^\mathrm{res}, \FF^\mathrm{res})$ has no killing inside, or in other words, its killing measure equals zero. We found that $(\EE,\FF)$ and $(\EE^\mathrm{res},\FF^\mathrm{res})$ share the same structure of regular subspaces, i.e. if $(\EE',\FF')$ is a regular subspace of $(\EE,\FF)$, then its resurrected Dirichlet form $(\EE^{\mathrm{res}'},\FF^{\mathrm{res}'})$ is also a regular subspace of $(\EE^\mathrm{res},\FF^\mathrm{res})$. The contrary fact is still right via the killing transform induced by the killing measure of $(\EE,\FF)$. 

In general, we are always curious about the essential factor in a Dirichlet form to produce non-trivial regular subspaces. 
In this paper, we shall use the killing transforms induced by MFs to explore the big jump part of $(\EE,\FF)$. Our main results indicate that 
the big jump part of $(\EE,\FF)$ is not essential for producing non-trivial  regular subspaces of $(\EE,\FF)$ either. However, we need to point out that the status of small jump is still an open problem in the considerations about the regular subspaces of a Dirichlet form.

More precisely, in \S\ref{SEC2}, we shall consider the killing transform induced by an MF on regular subspaces. Let $(\EE,\FF)$ be a regular Dirichlet form on $L^2(E,m)$ and $X$ its associated symmetric Hunt process. Denote all regular subspaces of $(\EE,\FF)$ by $\mathscr{R}$.  Set $M$ to be a fixed MF of $X$. Then $(\EE^M,\FF^M)$ is the associated Dirichlet form of subprocess $(X,M)$. When $(\EE^M,\FF^M)$ is regular, we also denote its all regular subspaces by $\mathscr{R}_M$. Our main theorems (Theorem~\ref{THM24} and \ref{THM25}) indicate that there exists a bijection 
\[
	T_M: \mathscr{R}\rightarrow \mathscr{R}_M
\]
between all regular subspaces of $(\EE,\FF)$ and $(\EE^M,\FF^M)$, and $T_M$ is actually the killing transform induced by an equal bivariate Revuz measure $\nu_M$. In \S\ref{SEC3}, we shall decompose the jumping measure of $(\EE,\FF)$ into a big jump part and a small jump part. The idea of this decomposition comes from It\^o's decomposition in cases of L\'evy processes. Particularly, the L\'evy measure is a big jump measure if and only if it is finite. Then we may naturally subtract the big jump part from $(\EE,\FF)$, and Theorem~\ref{THM33} proved that this procedure is exactly a combination of killing transform by some MF and resurrected transform. Hence we may deduce from \S\ref{SEC2} and \cite{LY15-2} that the big jump part is not a essential factor to produce non-trivial regular subspaces. Finally in \S\ref{SEC4}, we shall give some typical examples, say pure jump step processes, to illustrate that the pure big jump type Dirichlet forms have no proper regular subspaces. 

\section{Killing transforms induced by MFs on regular subspaces}\label{SEC2}

Let $E$ be a locally compact separable metric space and $m$ a fully supported Radon measure on $E$. We use $C_c(E)$ to denote all continuous functions with compact supports on $E$. Further let $(\EE,\FF)$ be a fixed regular symmetric Dirichlet form on $L^2(E,m)$ with $X$ being its associated $m$-symmetric Hunt process. Then $(\EE,\FF)$ admits the following Beurling-Deny decomposition: for any $u,v\in \FF\cap C_c(E)$, 
\begin{equation}\label{EQ2EUV}
\begin{aligned}
	\EE(u,v)&=\EE^{(c)}(u,v)\\&\quad +\int_{E\times E\setminus d} \left(u(x)-u(y)\right)\left(v(x)-v(y)\right) J(dxdy)+\int_E u(x)v(x)k(dx),
\end{aligned}\end{equation}
where $\EE^{(c)}$ is its strongly local part, $d$ is the diagonal of $E\times E$,  $J$ and $k$ are the jumping and killing measures. Without loss of generality, we always take the function in Dirichlet space $\FF$ as its quasi-continuous version.  Moreover, $\text{MF}(X)$ stands for all exact decreasing MFs of $X$ such that $M_0\equiv 1$ (if $M=(M_t)_{t\geq 0}$ 
represents the MF). For any $M\in \text{MF}(X)$, the subprocess $(X,M)$ is nearly symmetric relative to $m$ (see \cite{Y96}) and its associated Dirichlet form is given by \eqref{EQ1FMF}. 

\begin{definition}
	Given $M\in \text{MF}(X)$, $M$ is called $m$-symmetric if the subprocess $(X,M)$ is $m$-symmetric. 
\end{definition}

Denote the bivariate Revuz measure of $M$ by $\nu_M$ and 
\[
	\check{\nu}_M(f\otimes g):=\nu_M(g\otimes f)
\]
for any non-negative functions $f,g$. Then it follows from \eqref{EQ1FMF} and Theorem~6.11 of \cite{Y96} that $M$ is $m$-symmetric if and only if 
\[
	\nu_M=\check{\nu}_M. 
\]
Furthermore, the left and right marginal measures of $\nu_M$ equals, i.e. $\rho_M=\lambda_M$. We rewrite them as $\mu_M$, in other words, let $\mu_M:=\rho_M=\lambda_M$. The following lemma is natural, but we do not find it in other place. And the proof is not trivial.

\begin{lemma}\label{LM22}
	Assume $M\in \text{MF}(X)$ to be $m$-symmetric. If $\mu_M$ is Radon on $E$, then the associated Dirichlet form $(\EE^M,\FF^M)$ of subprocess $(X,M)$ is a regular Dirichlet form on $L^2(E,m)$. Furthermore, any special standard core of $(\EE,\FF)$ is still a special standard core of $(\EE^M,\FF^M)$. 
\end{lemma}
\begin{proof}
Clearly, it follows from \eqref{EQ1FMF} that
\begin{equation}\label{EQ2FMF}
 \FF^M=\FF\cap L^2(E,\mu_M). 
\end{equation}
From the property of bivariate Revuz measure (see (4.5c) of \cite{Y96}), we may obtain that $\mu_A$ is a Radon smooth measure relative to $X$. Thus the perturbed Dirichlet form $(\EE^{\mu_M},\FF^{\mu_M})$ of $(\EE,\FF)$ induced by $\mu_M$ (see \S6.1 of \cite{FOT11}) is given by
\begin{equation}
\begin{aligned}
	\FF^{\mu_M}&=\FF\cap L^2(E,\mu_M), \\
	\EE^{\mu_M}(u,v)&=\EE(u,v)+\int_E u(x)v(x)\mu_M(x), \quad u,v \in \FF^{\mu_M}. 
\end{aligned}
\end{equation}
In particular, $(\EE^{\mu_M},\FF^{\mu_M})$ is a regular Dirichlet form on $L^2(E,m)$, and any special standard core $\mathcal{C}$ of $(\EE,\FF)$ is still a special standard core of $(\EE^{\mu_M},\FF^{\mu_M})$. Note that $\FF^M=\FF^{\mu_M}$. Hence it suffices to prove that $\mathcal{C}$ is dense in $\FF^M$ with respect to the norm $\|\cdot \|_{\EE^M_1}$. In fact, assume that $(\EE,\FF)$ admits the Beurling-Deny decomposition \eqref{EQ2EUV}. Then it follows from Proposition~4.12 of \cite{Y96} that
\begin{equation}\label{EQ2NME}
	\nu_M|_{E\times E\setminus d}<2J. 
\end{equation}
Hence for any $u\in \mathcal{C}\subset \FF\cap C_c(E)$, we have
\begin{equation}\label{EQA1EMU}
\begin{aligned}
	\EE^M(u,u)=\EE^{(c)}&(u,u)+\int_{E\times E\setminus d} \big(u(x)-u(y)\big)^2(J-\frac{1}{2}\nu_M)(dxdy)\\&+\int_E u(x)^2 (k+\mu_M)(dx).
\end{aligned}	\end{equation}
Particularly,
\begin{equation}\label{EQA1EMUU}
\EE^M(u,u)\leq \EE^{\mu_M}(u,u),\quad u\in \mathcal{C}.
\end{equation}
Consequently, for any function $u\in \FF^M=\FF^{\mu_M}$, we may take a sequence of functions
\[
	\{u_n: n\geq 1\}\subset \mathcal{C}
	\]
such that when $n\rightarrow \infty$, $\EE^{\mu_M}_1(u_n-u,u_n-u)\rightarrow 0$. Apparently, we may deduce from \eqref{EQA1EMUU} that $\EE^M_1(u_n-u,u_n-u)\rightarrow 0$. In other words, $\mathcal{C}$ is also a special standard core of $(\EE^M, \FF^M)$. That completes the proof. 
\end{proof}
\begin{remark}\label{RM23}
	We may regard the killing transform induced by $M\in \text{MF}(X)$ as the combination of two killing transforms: the first one is decided by the part of $\nu_M$ outside the diagonal $d$, i.e. the second term in Beurling-Deny formula \eqref{EQA1EMU}, which kills lots of jumps in intuition; the second one adds killing inside, which is expressed in the third term of \eqref{EQA1EMU}. 
\end{remark}

Because of Lemma~\ref{LM22}, we denote all MFs $M\in \text{MF}(X)$ such that $M$ is $m$-symmetric and $\mu_M$ is Radon by $\overset{\circ}{\text{MF}}(X)$. For any $M\in \text{MF}(X)$, set 
\[
	M(\EE,\FF):=(\EE^M,\FF^M).
\]
In other words, we use $M(\cdot, \cdot)$ to stand for the killing transform induced by $M$. If $\nu_M$ is the bivariate Revuz measure of $M$, we also write 
\[
\nu_M(\EE,\FF):=M(\EE,\FF).
\]
In particular, if $M\in \overset{\circ}{\text{MF}}(X)$, then the Dirichlet form after killing transform is still regular on $L^2(E,m)$. Now let us consider the regular subspaces of $(\EE,\FF)$. We first assert that any regular subspace may be mapped uniquely to a regular subspace of $M(\EE,\FF)$ via the common feature (the equal bivariate Revuz measure). 

\begin{theorem}\label{THM24}
	Fix the regular Dirichlet form $(\EE,\FF)$ on $L^2(E,m)$, its associated symmetric Hunt process $X$ and an MF $M\in \overset{\circ}{\text{MF}}(X)$. Assume that $(\EE',\FF')$ is a regular subspace of $(\EE,\FF)$, i.e. $(\EE',\FF')\prec (\EE,\FF)$, its associated symmetric Hunt process is denoted by $X'$. Then there exists a unique MF $M'\in \overset{\circ}{\text{MF}}(X')$ such that
	\[
		M'(\EE',\FF')\prec M(\EE,\FF). 
	\]
Particularly, the bivariate Revuz measures of $M$ and $M'$ equals, i.e. $\nu_M=\nu_{M'}$. Here, the uniqueness of MF is in sense of $m$-equivalence (see Definition~2.1 of \cite{Y96}). 
\end{theorem}
\begin{proof}
Let $\nu:=\nu_M$. Then the left and right marginal measures of $\nu$ are both equal to $\mu_M$, which is a smooth Radon measure with respect to $(\EE,\FF)$. It follows from Remark~1 of \cite{LY15-2} that $\mu_M$ is also a smooth Radon measure with respect to $(\EE',\FF')$.  On the other hand, from Theorem~1 of \cite{LY15-2}, $(\EE',\FF')$ has the same jumping and killing measures as $(\EE,\FF)$. Note that \eqref{EQ2NME} indicates that 
\[
	\nu|_{E\times E\setminus d}<2J. 
\]
Then we can deduce that $\nu$ is a bivariate smooth measure with respect to $(\EE',\FF')$. We refer the definition of bivariate smooth measure to \S4 of \cite{Y96-2}. Hence it follows from Theorem~4.3 of \cite{Y96-2} that there exists an MF $M'\in \text{MF}(X')$ such that 
\[
	\nu=\nu_{M'},
\]
where $\nu_{M'}$ is the bivariate Revuz measure of $M'$ relative to $X'$. In particular, since $\nu$ is $m$-symmetric, and its left and right marginal measures are Radon on $E$, it follows that $M'\in \overset{\circ}{\text{MF}}(X')$. Therefore, from the definition of regular subspace, \eqref{EQ1FMF} and Lemma~\ref{LM22}, we can obtain that
\[
	M'(\EE',\FF')\prec M(\EE,\FF). 
\]

Finally, we shall prove the uniqueness of $M'$. Assume that $M^1,M^1\in \overset{\circ}{\text{MF}}(X')$ are two MFs that satisfy the above conditions. Denote their bivariate Revuz measures by $\nu^1,\nu^2$ respectively. Moreover, for $i=1,2$, set $\mu^i$ to be the left (right) marginal measure of $\nu^i$. Then from Theorem~1 of \cite{LY15-2} and \eqref{EQA1EMU}, we know that the jumping measure of $M^i(\EE',\FF')$ equals 
\[
	J-\frac{1}{2}\nu^i|_{E\times E\setminus d}=J-\frac{1}{2}\nu_M|_{E\times E\setminus d},
\]
and its killing measure equals 
\[
	k+\mu^i=k+\mu_M.
\]
In other words, outside the diagonal $d$ of $E\times E$, $\nu^1=\nu^2=\nu_M$, whereas $\mu^1=\mu^2=\mu_M$. Therefore,
\[
	\nu^1=\nu^2=\nu_M.
\]
From Theorem~6.3 of \cite{Y96}, we can obtain that $M^1$ and $M^2$ are $m$-equivalent. That completes the proof. 
\end{proof}

Denote all regular subspaces of $(\EE,\FF)$ and $M(\EE,\FF)$ by $\mathscr{R}$ and $\mathscr{R}_M$ respectively. Then Theorem~\ref{THM24} implies that the following mapping
\begin{equation}\label{EQ2TMR}
	T_M: \mathscr{R}\rightarrow \mathscr{R}_M,\quad  (\EE',\FF')\mapsto \nu_M(\EE',\FF')
\end{equation}
is well-defined and injective. The following theorem indicates that it is also a surjection. In other words, $T_M$ is a bridge with the common feature $\nu_M$ between all regular subspaces of $(\EE,\FF)$ and $M(\EE,\FF)$. 

\begin{theorem}\label{THM25}
	Fix a regular Dirichlet form $(\EE,\FF)$, its associated symmetric Hunt process $X$ and an MF $M\in \overset{\circ}{\text{MF}}(X)$. Assume that $(\mathcal{A},\mathcal{G})$ is a regular subspace of $M(\EE,\FF)$, i.e. $(\mathcal{A},\mathcal{G})\prec M(\EE,\FF)$. Then there always exist a unique regular subspace $(\EE',\FF')$ of $(\EE,\FF)$ and a unique MF $M'\in \overset{\circ}{\text{MF}}(X')$ such that 
	\[
		(\mathcal{A},\mathcal{G})=M'(\EE',\FF'),
	\]
	where $X'$ is the associated symmetric Hunt process of $(\EE',\FF')$. Furthermore, the bivariate Revuz measure of $M'$ is equal to $\nu_M$, i.e. $\nu_{M'}=\nu_M$.  
\end{theorem}
\begin{proof}
We first prove the existence of $(\EE',\FF')$ and $M'$. Let
\[
	\mathcal{C}:=\mathcal{G}\cap C_c(E),
\]
which is a special standard core of $(\mathcal{A},\mathcal{G})$. Since $\mathcal{C}\subset \mathcal{G}\subset \FF^M\subset \FF$, it follows that the quadratic form $(\EE,\mathcal{C})$ is closable. Denote the smallest closed extension of $\mathcal{C}$ in $\FF$ with the norm $\|\cdot\|_{\EE_1}$ by $\FF'$. For any $u,v\in \FF'$, define
\[
	\EE'(u,v):=\EE(u,v). 
\]
Then from Theorem~3.1.1 of \cite{FOT11}, we may deduce that $(\EE',\FF')$ is a regular Dirichlet form on $L^2(E,m)$. Particularly,
\[
	(\EE',\FF')\prec (\EE,\FF). 
\]
Set $X'$ to be the associated Hunt process of $(\EE',\FF')$. We know from Theorem~\ref{THM24} that there exists a unique MF $M'\in \overset{\circ}{\text{MF}}(X')$ such that 
\[
	M'(\EE',\FF')\prec M(\EE,\FF),\quad \nu_{M'}=\nu_{M},\quad \mu_{M'}=\mu_{M}. 
\]
Thus we only need to prove that $M'(\EE',\FF')= (\mathcal{A},\mathcal{G})$. Indeed, since $\mathcal{C}$ is a special standard core of $(\EE',\FF')$, it follows from Lemma~\ref{LM22} that $\mathcal{C}$ is also a special standard core of $M'(\EE',\FF')$. In particular, for any $u,v\in \mathcal{C}$,
\[
	\EE'^{M'}(u,v)=\EE'(u,v)+\nu_{M'}(u\otimes v)=\EE(u,v)+\nu_M(u\otimes v)=\EE^M(u,v)=\mathcal{A}(u,v). 
\]
Note that $\mathcal{C}$ is also a special standard core of $(\mathcal{A},\mathcal{G})$, whereas $\mathcal{A}|_{\mathcal{C}\times \mathcal{C}}=\EE'^{M'}|_{\mathcal{C}\times \mathcal{C}}$. Therefore, we may obtain that $M'(\EE',\FF')=(\mathcal{A},\mathcal{G})$. 

Finally, let us prove the uniqueness of $(\EE',\FF')$ and $M'$. Assume that $(\EE^1,\FF^1), M^1$ and $(\EE^2,\FF^2), M^2$ are two groups of regular subspace and MF that satisfy the conditions. Since
\[
	M^1(\EE^1,\FF^1)=M^2(\EE^2,\FF^2)=(\mathcal{A},\mathcal{G}), 
\]
it follows from \eqref{EQ2FMF} that 
\[
	\FF^1\cap L^2(E,\mu_{M^1})=\FF^2\cap L^2(E,\mu_{M^2})=\mathcal{G}. 
\]
Hence 
\[
	\FF^1\cap C_c(E) =\FF^2\cap C_c(E) = \mathcal{G}\cap C_c(E). 
\]
In other words, the regular subspaces $(\EE^1,\FF^1)$ and $(\EE^2,\FF^2)$ of $(\EE,\FF)$ have the common special standard core $\mathcal{G}\cap C_c(E)$. Apparently, we have $(\EE^1,\FF^1)=(\EE^2,\FF^2)$. Then the uniqueness of MF is directly from Theorem~\ref{THM24}. That completes the proof. 
\end{proof}

\begin{remark}\label{RM26}
	In this note, we shall consider a special case of Theorem~\ref{THM24} and Theorem~\ref{THM25}. Without loss of generality, assume that $(\EE,\FF)$ has no killing inside, in other words, the killing measure $k$ in Beurling-Deny formula \eqref{EQ2EUV} equals zero. Let $A$ be a positive continuous additive functional of $X$ whose Revuz measure $\mu_A$ is Radon on $E$ (see \S6.1 of \cite{FOT11}). Further set 
	\[
		M_t:=\exp\{-A_t\},\quad t\geq 0.
	\]
One may easily check that $M\in \overset{\circ}{\text{MF}}(X)$ and the subprocess $(X,M)$ corresponds to the perturbed Dirichlet form $(\EE^{\mu_A},\FF^{\mu_A})$ of $(\EE,\FF)$ with respect to $\mu_A$. Denote 
\[
	\mu_A(\EE,\FF):=(\EE^{\mu_A},\FF^{\mu_A})=M(\EE,\FF),\quad \mathscr{R}_A:=\mathscr{R}_M.  
\]
The above two theorems indicate that
\[
 	T_A: \mathscr{R}\rightarrow \mathscr{R}_A,\quad (\EE',\FF')\mapsto \mu_A(\EE',\FF')
\]	
is a bijection. Particularly, the regular subspace $(\EE',\FF')$ of $(\EE,\FF)$ is actually the resurrected Dirichlet form of regular subspace $\mu_A(\EE',\FF')$ of $\mu_A(\EE,\FF)$. Note that this fact has been illustrated in \S2.2.3 of \cite{LY15-2}. 
\end{remark}

At the end of this section, we shall describe the inverse mapping of $T_M$, which is given by \eqref{EQ2TMR}. Without loss of generality, we still assume that $X$ has no killing inside, i.e. $k=0$. Let $S_M:=T_M^{-1}$. 

In Remark~\ref{RM23}, we noted that $T_M$ may be decomposed into two parts: one kills some jumps and the other is the perturbation induced by $\mu_M$. Clearly, the inverse of perturbation is exactly the resurrected transform. More precisely, taking any regular subspace $(\mathcal{A},\mathcal{G})\in \mathscr{R}_M$, there exists a unique regular subspace $(\EE',\FF')\in \mathscr{R}$ such that 
\begin{equation}\label{EQ2AGM}
(\mathcal{A},\mathcal{G})=\nu_M(\EE',\FF').
\end{equation}
 Particularly, $(\mathcal{A},\mathcal{G})$ admits the following Beurling-Deny decomposition: for any $u\in \mathcal{G}\cap C_c(E)$, 
\[
\begin{aligned}
	\mathcal{A}(u,u)=\EE^{(c)}&(u,u)+\int_{E\times E\setminus d} \big(u(x)-u(y)\big)^2(J-\frac{1}{2}\nu_M)(dxdy)\\&+\int_E u(x)^2 \mu_M(dx).
\end{aligned}	\]
Thus the resurrected Dirichlet form $(\mathcal{A}^\mathrm{res},\mathcal{G}^\mathrm{res})$ of $(\mathcal{A},\mathcal{G})$ satisfies $\mathcal{G}^\mathrm{res}\cap C_c(E)=\mathcal{G}\cap C_c(E)$, and for any $u\in \mathcal{G}^\mathrm{res}\cap C_c(E)$,
\begin{equation}\label{EQ2ARU}
	\mathcal{A}^\mathrm{res}(u,u)=\EE^{(c)}(u,u)+\int_{E\times E\setminus d} \big(u(x)-u(y)\big)^2(J-\frac{1}{2}\nu_M)(dxdy).
\end{equation}

Now we shall define an add-jump transform relative to $\nu_M$ on $(\mathcal{A}^\mathrm{res},\mathcal{G}^\mathrm{res})$ to offset the first part of $T_M$. That is, for any $u,v\in \mathcal{G}\cap C_c(E)$, 
\begin{equation}\label{EQ2ARU2}
	\mathcal{A}^{\mathrm{res},\nu_M}(u,v):=\mathcal{A}^\mathrm{res}(u,v)+\frac{1}{2}\int_{E\times E\setminus d}\left( u(x)-u(y)\right)\left(v(x)-v(y)\right)\nu_M(dxdy).
\end{equation}
We assert that the quadratic form $\left(\mathcal{A}^{\mathrm{res},\nu_M}, \mathcal{G}\cap C_c(E)\right)$ is closable on $L^2(E,m)$. Its smallest closed extension is denoted by $\left(\mathcal{A}^{\mathrm{res},\nu_M}, \mathcal{G}^{\mathrm{res},\nu_M}\right)$, which is called the add-jump transformed Dirichlet form relative to $\nu_M$ of resurrected Dirichlet form $(\mathcal{A}^\mathrm{res},\mathcal{G}^\mathrm{res})$. 

\begin{proposition}
	For any $(\mathcal{A},\mathcal{G})\in \mathscr{R}_M$, the quadratic form $\left(\mathcal{A}^{\mathrm{res},\nu_M}, \mathcal{G}\cap C_c(E)\right)$ is closable on $L^2(E,m)$. Furthermore, its smallest closed extension equals $(\EE',\FF')$, i.e.
	\begin{equation}
		\left(\mathcal{A}^{\mathrm{res},\nu_M}, \mathcal{G}^{\mathrm{res},\nu_M}\right)=(\EE',\FF'),
	\end{equation}
where $(\EE',\FF')=S_M(\mathcal{A},\mathcal{G})$ is the regular subspace of $(\EE,\FF)$ in \eqref{EQ2AGM}. 
\end{proposition}
\begin{proof}
From the proof of Theorem~\ref{THM25}, we know that $\mathcal{G}\cap C_c(E)$ is a special standard core of $(\EE',\FF')$. On the other hand, it follows from \eqref{EQ2ARU} and \eqref{EQ2ARU2} that for any $u,v\in \mathcal{G}\cap C_c(E)$, 
\[
	\mathcal{A}^{\mathrm{res},\nu_M}(u,v)=\EE(u,v)=\EE'(u,v). 
\]
Then clearly $\left(\mathcal{A}^{\mathrm{res},\nu_M}, \mathcal{G}\cap C_c(E)\right)$ is closable and its smallest closed extension equals $(\EE',\FF')$. That completes the proof. 
\end{proof}

In a word, the inverse mapping $S_M$ of $T_M$ can be also decomposed into two steps: first one is the resurrected transform and the second one is the add-jump transform relative to $\nu_M$. 

\section{Big jump parts of regular subspaces}\label{SEC3}

In this section, we shall reconsider the regular subspaces of $(\EE,\FF)$, which is a fixed regular Dirichlet form on $L^2(E,m)$ and admits the Beurling-Deny decomposition \eqref{EQ2EUV}. The class of all regular subspaces of $(\EE,\FF)$ is still denoted by $\mathscr{R}$.  In our previous work \cite{LY15-2}, we have already explained that the killing part of $(\EE,\FF)$ is not essential for the constitution of $\mathscr{R}$ (see also Remark~\ref{RM26}). Thus without loss of generality, we may always assume that $k=0$. 

Now let us consider the big jump part of $(\EE,\FF)$. Note that the jumping measure $J$ is a symmetric Radon measure on $E\times E\setminus d$. Further assume that $J$ can be written as a sum of two positive symmetric Radon measures:
\begin{equation}\label{EQ3JJB}
	J=J_\mathrm{b}+J_\mathrm{s},
\end{equation}
where $J_\mathrm{b}$ satisfies that its marginal measure $\mu_\mathrm{b}(dx):=J_\mathrm{b}(dx\times E\setminus d)$ is a Radon measure on $E$. Then $J_\mathrm{b}$ is called the \emph{big jump part} of $J$, $J_\mathrm{s}$ is called the \emph{small jump part} of $J$ if its marginal measure is not Radon.  Particularly, if the marginal measure of $J$ is Radon on $E$, then $J$ is called a \emph{big jump measure}. The idea of decomposition \eqref{EQ3JJB} comes from the It\^o's decomposition of L\'evy process, and it may not be unique if exists. More precisely, we have the following examples.

\begin{example}
	Let $E$ be a finite dimensional space and $\mathtt{d}$ the metric on $E$. Take a small enough constant $\delta>0$ and let
	\[
		D_\delta:=\left\{ (x,y)\in E\times E: \delta<\mathtt{d}(x,y)<\frac{1}{\delta}\right\}. 
	\] 
	Clearly, $D_\delta\subset E\times E\setminus d$. Further set 
	\[
		J_\mathrm{b}:=J|_{D_\delta},\quad J_\mathrm{s}:=J-J_\mathrm{s}. 
	\]
We claim that $J_\mathrm{b}$ is the big jump part of $J$. In fact, for any compact subset $K$ of $E$, since $E$ is a finite dimensional space, it follows that $(K\times E)\cap D_\delta$ is a compact subset of $E\times E\setminus d$. But $J$ is Radon on $E\times E\setminus d$, thus we have
\[
	\mu_\mathrm{b}(K):=J_\mathrm{b}(K\times E\setminus d)=J((K\times E)\cap D_\delta)<\infty. 
\]
Furthermore, $J_\mathrm{s}$ is a small jump if and only if $J$ is not a big jump measure. 

Next, assume $E=\mathbf{R}^k$ for some positive integer $k$ and the jumping measure $J$ is induced by a symmetric L\'evy measure $\mathtt{n}$ on $\mathbf{R}^k\setminus \{0\}$. In other words, $\mathtt{n}$ is a symmetric measure on $\mathbf{R}^k\setminus \{0\}$ such that 
\[
	\int_{\mathbf{R}^k\setminus \{0\}} (1\wedge |x|^2)\mathtt{n}(dx)<\infty,
\]
and 
\[
	J(dxdy)=\mathtt{n}(dy-x)dx,
\]
where $dx$ is the Lebesgue measure on $\mathbf{R}^k$. Note that the L\'evy measure $\mathtt{n}$ is a big jump measure if and only if $\mathtt{n}$ is a finite measure. In general, take an arbitrary constant $\delta>0$, and let
\[
	D_\delta:=\{(x,y)\in \mathbf{R}^k\times \mathbf{R}^k: |x-y|>\delta\}. 
\]
Define 
\[
	J_\mathrm{b}:=J|_{D_\delta},\quad J_\mathrm{s}:=J-J|_{D_\delta}.
\]
We assert $J_\mathrm{b}$ is the big jump part of $J$. In fact, we may easily obtain that 
\[
	\mathtt{n}\left(\{y\in \mathbf{R}^k: |y|>\delta\}\right)<\infty,
\]
and denote this finite constant by $C_\delta$. Hence for any compact subset $K\subset \mathbf{R}^k$, we have
\[
	J_\mathrm{b}(K\times \mathbf{R}^k\setminus d)=\int_K dx\int_{\{y\in \mathbf{R}^k: |y-x|>\delta\}} \mathtt{n}(dy-x)=C_\delta\cdot \int_K dx.
\]
That implies the marginal measure $\mu_\mathrm{b}$ of $J_\mathrm{b}$ is a multiple of Lebesgue measure, which is clearly a Radon measure. This is actually a part of It\^o's decomposition for L\'evy process.

Note that for the cases of L\'evy processes, the big jump measure usually corresponds to a compound Poisson process, and the typical example of non-big jump measure is the symmetric $\alpha$-stable process for some constant $0<\alpha<2$.
\end{example}

For any $u,v\in \FF\cap C_c(E)$, define a new form
\[
	\EE^\mathrm{s}(u,v)=\EE^{(c)}(u,v)+\int_{E\times E\setminus d} \left(u(x)-u(y)\right)\left(v(x)-v(y)\right)J_\mathrm{s}(dxdy). 
\]
Roughly speaking, $\EE^\mathrm{s}$ is the small jump part of $\EE$. The following lemma asserts that the quadratic form $\left(\EE^\mathrm{s},\FF\cap C_c(E)\right)$ is a closable Markovian symmetric form on $L^2(E,m)$.

\begin{lemma}\label{LM32}
	The quadratic form $(\EE^\mathrm{s}, \FF\cap C_c(E))$ is a closable Markovian symmetric form on $L^2(E,m)$. 
\end{lemma}
\begin{proof}
Let $\mathcal{C}:=\FF\cap C_c(E)$, and $\mu_\mathrm{b}$ be the marginal measure of $J_\mathrm{b}$. Then $\mu_\mathrm{b}$ is a Radon smooth measure relative to $(\EE,\FF)$. Define another form $\tilde{\EE}^\mathrm{s}$ for any $u,v\in \mathcal{C}$:
\[
	\tilde{\EE}^\mathrm{s}(u,v):=\EE^\mathrm{s}(u,v)+4\int_E u(x)v(x)\mu_\mathrm{b}(dx). 
\]
One may easily check that 
\begin{equation}\label{EQ3ESU}
	\EE^\mathrm{s}(u,u)\leq \EE(u,u)\leq \tilde{\EE}^\mathrm{s}(u,u)
\end{equation} 
for any $u\in \mathcal{C}$.

We assert that $(\tilde{\EE}^\mathrm{s}, \mathcal{C})$ is a closable Markovian form on $L^2(E,m)$. Clearly we only need to prove its closable property. In fact, take an $\tilde{\EE}^\mathrm{s}$-Cauchy sequence $\{u_n:n\geq 1\}\subset \mathcal{C}$ such that $u_n\rightarrow 0$ as $n\rightarrow \infty$ in $L^2(E,m)$. Then it follows from \eqref{EQ3ESU} that it is also $\EE$-Cauchy and hence $u_n\rightarrow 0$ with the norm $\|\cdot\|_{\EE_1}$. Particularly, a subsequence of $\{u_n: n\geq 1\}$ is $\EE$-q.e. convergent to $0$. Since $\{u_n:n\geq 1\}$ is also a Cauchy sequence in $L^2(E,\mu_\mathrm{b})$ and $\mu_\mathrm{b}$ does not charge any $\EE$-polar set, thus we can deduce that $u_n$ is also convergent to $0$ in $L^2(E,\mu_\mathrm{b})$ as $n\rightarrow \infty$. That implies the closable property of $(\tilde{\EE}^\mathrm{s}, \mathcal{C})$. 

Finally, $(\EE^\mathrm{s},\mathcal{C})$ is actually the resurrected form of $(\tilde{\EE}^\mathrm{s}, \mathcal{C})$. Hence from Theorem~5.2.17 of \cite{CF12}, we may obtain that $(\EE^\mathrm{s}, \FF\cap C_c(E))$ is a closable Markovian form. That completes the proof. 
\end{proof}

Naturally, we denote the smallest closed extension of $(\EE^\mathrm{s}, \FF\cap C_c(E))$ by $(\EE^\mathrm{s}, \FF^\mathrm{s})$. Particularly, $(\EE^\mathrm{s},\FF^\mathrm{s})$ is a regular Dirichlet form on $L^2(E,m)$. Roughly speaking, we may regard $(\EE^\mathrm{s},\FF^\mathrm{s})$ as the rest of $(\EE,\FF)$ after subtracting the big jump part. The following theorem is the main result of this section.

\begin{theorem}\label{THM33}
There exists an MF $M\in \overset{\circ}{\text{MF}}(X)$ such that $(\EE^\mathrm{s},\FF^\mathrm{s})$ is the resurrected Dirichlet form of $M(\EE,\FF)$. 
\end{theorem}
\begin{proof}
We still set $\mathcal{C}:= \FF\cap C_c(E)$. Let $\nu:=2J_\mathrm{b}$. Since $J_\mathrm{b}\leq J$, it follows that $\nu$ is a bivariate smooth measure (see \S4 of \cite{Y96-2}). Thus from Theorem~4.3 of \cite{Y96-2}, we may deduce that there exists an MF $M\in \text{MF}(X)$ such that 
\[
	\nu_M=\nu,
\]
where $\nu_M$ is the bivariate Revuz measure of $M$. Moreover, from the definition of $J_\mathrm{b}$, we also have $M\in \overset{\circ}{\text{MF}}(X)$. Then it follows from Lemma~\ref{LM22} that $M(\EE,\FF)$ is a regular Dirichlet form and $\mathcal{C}$ is its special standard core. Furthermore,  \eqref{EQA1EMU} indicates that for any $u\in \mathcal{C}$, 
\[
\begin{aligned}
	\EE^M&(u,u)\\ &=\EE^{(c)}(u,u)+\int_{E\times E\setminus d} \big(u(x)-u(y)\big)^2(J-\frac{1}{2}\nu_M)(dxdy)+\int_E u(x)^2 \mu_M(dx) \\
		&= \EE^{(c)}(u,u)+\int_{E\times E\setminus d} \big(u(x)-u(y)\big)^2J_\mathrm{s}(dxdy)+\int_E u(x)^2 \mu_M(dx) \\
		&=\EE^\mathrm{s}(u,u)+\int_E u(x)^2 \mu_M(dx).
\end{aligned}\]
Denote the resurrected Dirichlet form of $M(\EE,\FF)$ by $\left(\EE^{M,\mathrm{res}}, \FF^{M,\mathrm{res}}\right)$.  It follows from Theorem~5.2.17 of \cite{CF12} that $\mathcal{C}$ is also a special standard core of $\left(\EE^{M,\mathrm{res}}, \FF^{M,\mathrm{res}}\right)$, and for any $u,v\in \mathcal{C}$, 
\[
	\EE^{M,\mathrm{res}}(u,v)=\EE^\mathrm{s}(u,v). 
\]
Therefore, from Lemma~\ref{LM32}, we can obtain that $\left(\EE^{M,\mathrm{res}}, \FF^{M,\mathrm{res}}\right)=(\EE^\mathrm{s},\FF^\mathrm{s})$. That completes the proof. 
\end{proof} 

The above theorem implies that the difference between $(\EE,\FF)$ and $(\EE^\mathrm{s},\FF^\mathrm{s})$ is a combination of killing transform induced by $M$ (or $\nu_M=2J_\mathrm{b}$) and resurrected transform. Denote all regular subspaces of $(\EE,\FF)$ and $(\EE^\mathrm{s},\FF^\mathrm{s})$ by $\mathscr{R}$ and $\mathscr{R}_\mathrm{s}$. Let $T_\mathrm{s}$ be the composite transform of killing transform induced by $2J_\mathrm{b}$ and resurrected transform. Then we may deduce that 
\[
	T_\mathrm{s}: \mathscr{R}\rightarrow \mathscr{R}_\mathrm{s}
\] 
is a bijection. The inverse mapping of $T_\mathrm{s}$ is actually the add-jump transform relative to $2J_\mathrm{b}$, which is outlined at the end of \S\ref{SEC2}. 

In a word, the big jump part of $J$ is not essential for the questions about regular subspaces of $(\EE,\FF)$ either. This fact is very similar to the discussions about the killing part of $(\EE,\FF)$.

\section{Examples: pure jump step processes}\label{SEC4}

In the end, we shall give some examples of pure big-jump processes to illustrate that they have no proper regular subspaces. 

We refer the general introduction to pure jump step processes to Chapter I \S12 of \cite{BG68}. Under symmetric settings, \S2.2.1 of \cite{CF12} constructed their associated Dirichlet forms. Let $E$ be a locally compact separable metric space, and $Q(x,dy)$ a probability kernel on $(E,\mathcal{B}(E))$, where $\mathcal{B}(E)$ is the Borel $\sigma$-algebra on $E$. Assume that $Q(x,\{x\})=0$ for every $x\in E$. Further set $\lambda(x)$ to be a Borel measurable function on $E$ such that $0<\lambda(x)<\infty$. Then $Q$ is called the road map of pure jump step process $X$, $\lambda$ is called its speed function, and $X$ is also written as $X^{\lambda, Q}$. We refer more details to \cite{BG68} and \cite{CF12}. 

We always assume that $X^{\lambda, Q}$ is symmetric. More precisely, assume that there exists a $\sigma$-finite measure $m_0$ fully supported on $E$ such that
\begin{equation}\label{QSYM}
    Q(x,dy)m_0(dx)=Q(y,dx)m_0(dy).
\end{equation}
Then $m_0$ is called a symmetrizing measure of $Q$. Let
\begin{equation}
    m(dx):=\frac{1}{\lambda(x)}m_0(dx),
\end{equation}
which is called the speed measure for $X^{\lambda,Q}$. Further assume that $m$ is Radon. Note that under above assumptions, $X^{\lambda,Q}$ is $m$-symmetric.

\subsection{Bounded speed function}\label{SEC41}

When $\lambda$ is bounded, one may easily check that (see also Theorem~2.2.3 of \cite{CF12}) the associated  Dirichlet form of $X$ on $L^2(E,m)$ is
\begin{equation}\label{BOUND}
\begin{aligned}
	\mathcal{F} & =L^2(E,m), \\
    \mathcal{E}(u,v) &=\frac{1}{2}\int_{E\times E}(u(x)-u(y))(v(x)-v(y))Q(x,dy)\lambda(x)m(dx) \\&\qquad \qquad  +\int_Eu(x)v(x)(1-Q(x,E))\lambda(x)m(dx), \quad u,v \in \FF. 
      \end{aligned}
    \end{equation}
Clearly $(\EE,\FF)$ is regular on $L^2(E,m)$. 

\begin{proposition}\label{AQS}
     Assume that $Q$ satisfies \eqref{QSYM}, and $\lambda$ is bounded. If $(\mathcal{E}',\mathcal{F}')$ is a regular subspace of $(\EE,\FF)$,  then $(\mathcal{E}',\mathcal{F}')=(\mathcal{E},\mathcal{F})$
\end{proposition}
\begin{proof}
    Let $M$ be a positive number such that $|\lambda(x)|<M$ for any $x\in E$. It follows from the symmetry of $Q$ and $ \left(u(x)-u(y)\right)^2 \leq 2\left(u(x)^2+u(y)^2\right)$ that for any $u\in \FF$,
    \[
     \begin{aligned}
        \mathcal{E}&(u,u)\\&\leq \int_{E\times E}\left(u(x)^2+u(y)^2\right)Q(x,dy)\lambda(x)m(dx) +\int_Eu(x)^2(1-Q(x,E))\lambda(x)m(dx)\\
                        &= 2\int_{E\times E}u(x)^2Q(x,dy)\lambda(x)m(dx) +\int_Eu(x)^2(1-Q(x,E))\lambda(x)m(dx)\\
                        &\leq 2M\int_Eu(x)^2m(dx).
    \end{aligned}\]
 That indicates that the norm $\|\cdot\|_{\mathcal{E}_1}$ is equivalent to the $L^2(E,m)$-norm on $\FF$. Since $\mathcal{F}'\cap C_c(E)$ is dense in $C_c(E)$ with uniform norm, we may also deduce that $\FF'\cap C_c(E)$ is dense in $\FF$ with the norm $\|\cdot\|_{\EE_1}$. Particularly, we have $(\EE',\FF')=(\EE,\FF)$. That completes the proof. 
\end{proof}

Now, let us consider the examples on Euclidean space. More precisely,  $E=\mathbf{R}^k$, $Q$ is a spatial homogeneous and conservative Markov kernel, i.e. $Q(x,dy)=Q(x-z,dy-z)$ and $Q(z,\mathbf{R}^k)=1$ for any $z\in \mathbf{R}^k$. Further assume that  $\lambda(x)$ is a constant function, i.e. $\lambda(x)\equiv \lambda_0$ for some constant $\lambda_0$. Then $X^{\lambda,Q}$ is exactly a compound Poisson process on $\mathbf{R}^k$ with parameter $\lambda_0$ and probability distribution $Q$. In other words, 
\[
	X^{\lambda, Q}_t=x+\sum_{n=1}^{N_t}\xi_n,\quad t\geq 0, x\in \mathbf{R}^k,
\]
where $x$ is the starting point, $N$ is the standard Poisson process with parameter $\lambda_0$, and $\{\xi_n:n\geq 1\}$ is a sequence of i.i.d random variables who share the common law $Q(0, \cdot)$. From Proposition~\ref{AQS}, we can directly obtain the following corollary.

\begin{corollary}\label{ASC}
    Any symmetric compound Poisson process has no proper regular subspaces.
\end{corollary}

\subsection{Locally integrable speed function}\label{SEC42}

Next, we assume that $\lambda \in L^1_{\text{\text{loc}}}(E,m)$. Note that this is equivalent to that $m_0$ is Radon.

Let $Z$ be the $m_0$-symmetric regular step process on $E$ with speed function $1$ and road map $Q$, where $m_0(dx)=\lambda(x)m(dx)$. Then  $X^{\lambda,Q}$ is a time change of $Z$ by $\tau_t:= \inf\{s:A_s>t\}$, where
\[
	A_s=\int_0^s\frac{1}{\lambda(Z_r)}dr, \quad s\geq 0.
\]
In other words, $X^{\lambda,Q}_t=Z_{\tau_t}$ for any $t\geq 0$. Denote the associated Dirichlet form of $Z$ and its extended Dirichlet space by $(\mathcal{E}^Z,\mathcal{F}^Z)$ and $\mathcal{F}^Z_e$. In \S\ref{SEC41}, we introduced that $\mathcal{F}^Z=L^2(E,m_0)$ and $\mathcal{E}^Z$ is given by \eqref{BOUND}. Moreover $(\mathcal{E}^Z,\mathcal{F}^Z)$ is regular. Note that $\mathcal{F}^Z_e$ is a linear subspace of $\mathcal{G}$ containing $L^2(E,m_0)$, where
\[
  \mathcal{G}= \left\{u: \int_{E\times E}(u(x)-u(y))^2Q(x,dy)m_0(dx) +\int_E u(x)^2(1-Q(x,E))m_0(dx)<\infty \right\}.
\]
In particular, $\mathcal{F}^Z_e=\mathcal{G}$ holds if and only if $Z$ is recurrent. It is known that the associated Dirichlet form of $X^{\lambda,Q}$ on $L^2(E,m)$  is
\begin{equation}\label{UNBOUND}
\begin{aligned}
        \mathcal{F} &= L^2(E,m)\cap \mathcal{F}^Z_e, \\
        \mathcal{E}(u,v) &=\frac{1}{2}\int_{E\times E}(u(x)-u(y))(v(x)-v(y))Q(x,dy)\lambda(x)m(dx) \\&\qquad \qquad +\int_Eu(x)v(x)(1-Q(x,E))\lambda(x)m(dx),\quad u,v\in \FF. 
\end{aligned}
    \end{equation}
Since \eqref{UNBOUND} is the time-changed Dirichlet form of  $(\mathcal{E}^Z,\mathcal{F}^Z)$ with respect to $m$, we may deduce that $(\EE,\FF)$ is regular on $L^2(E,m)$. Finally, we can also prove that $(\EE,\FF)$ has no proper regular subspaces. That is similar to Proposition~\ref{AQS}.


\begin{thebibliography}{}



\bibitem{BG68}
Blumenthal, R.M., Getoor, R.: \textsc{Markov processes and potential theory.} Academic Press, New York-London (1968).

\bibitem{CF12}
Chen, Z.-Q., Fukushima, M.: \textsc{Symmetric Markov processes, time change, and boundary theory.} Princeton University Press, Princeton, NJ (2012).

\bibitem{FFY05}
Fang, X., Fukushima, M., Ying, J.: \textit{On regular Dirichlet subspaces of $H^1(I)$ and associated linear diffusions.} Osaka J. Math. 42, 27-41 (2005).

\bibitem{FOT11}
Fukushima, M., Oshima, Y., Takeda, M.: \textsc{Dirichlet forms and symmetric Markov processes.} Walter de Gruyter \& Co., Berlin (2011).

\bibitem{LY15}
Li, L., Ying, J.: \textit{Bivariate Revuz measures and the Feynman-Kac formula on semi-Dirichlet forms.} Potential Anal. 42, 775-808 (2015).

\bibitem{LY15-2}
Li, L., Ying, J.: \textit{Regular subspaces of Dirichlet forms.} In: Festschrift Masatoshi Fukushima. pp. 397-420. World Scientific (2015).

\bibitem{Y96}
Ying, J.: \textit{Bivariate Revuz measures and the Feynman-Kac formula.} Ann. Inst. H. Poincar\'e Probab. Statist. 32, 251-287 (1996).

\bibitem{Y96-2}
Ying, J.: \textit{Killing and subordination.} Proc. Amer. Math. Soc. 124, 2215-2222 (1996).
\end{thebibliography}
\end{document}